\documentclass{article}
\usepackage{hyperref}
\usepackage[utf8]{inputenc}
\author{Rikhav Shah
}
\date{December 2018}
\usepackage[numbers]{natbib}
\usepackage{graphicx}
\usepackage{geometry}
\usepackage{amsmath,amsthm,amssymb,graphicx,mathtools,tikz}
\newtheorem{theorem}{Theorem}[section]
\newtheorem{corollary}[theorem]{Corollary}
\newtheorem{lemma}[theorem]{Lemma}
\newtheorem{definition}{Definition}[section]
\newcommand{\row}{\mathbf{r}}
\newcommand{\srow}{\mathbf{s}}
\newcommand{\vecv}{\mathbf{v}}

\newcommand{\detset}{\mathcal{D}_n}

\newcommand{\ldash}{\textrm{---}\,\,}
\newcommand{\rdash}{\,\textrm{---}}
\newcommand{\cdash}{\textrm{---}}

\begin{document}
\title{Determinants of binary matrices achieve every integral value up to $\Omega(2^n/n)$}
\maketitle
\begin{abstract}
    This work shows that the smallest natural number $d_n$ that is not the determinant of some $n\times n$ binary matrix is at least $c\,2^n/n$ for $c=1/201$.  That same quantity naturally lower bounds the number of distinct integers $D_n$ which can be written as the determinant of some $n\times n$ binary matrix.  This asymptotically improves the previous result of $d_n=\Omega(1.618^n)$ and slightly improves the previous result of $D_n\ge 2^n/g(n)$ for a particular $g(n)=\omega(n^2)$ function.
\end{abstract}

\section{Introduction}
We take a binary matrix $M$ to mean a matrix with all entries in $\{0,1\}$.
We investigate the range of the determinant, i.e. $\detset=\{\det(M):M\in \{0,1\}^{n\times n}\}$
\footnote{Many papers take entries of binary matrices to be $\pm 1$.
There exists a one-to-one correspondence between matrices with entries in \{-1,1\} of size $n+1$ and matrices with entries in \{0,1\} of size $n$, such that the determinants of corresponding matrices are off by a constant factor of $(-2)^n$.
In this way, results pertaining to one type of matrix easily carry over to results about the other.}.
Much work has sought to characterize $\detset$.
An old conjecture stated that $\detset$ is a set of consecutive integers.  A proof that $n=7$ is a counter-example to this conjecture was published in 1969 by Metropolis \cite{metropolis}, and an apparently independent proof was published by Carigen \cite{posed} in 1990.
Carigen \cite{posed} also provides exact values of $\detset$ for $n\le7$, and large subsets of $\detset$ for $8\le n\le10$.  In 2004, Orrick \cite{n15} finds exact values for $n=8,10$ and a larger subset for $n=9$.  Despite the aforementioned conjecture being false, it seems empirically that $\detset$ does contain a large set of consecutive integers centered at $0$, so one may ask what the smallest natural number $d_n$ \textit{not} in $\detset$ is.  For large $n$, the best lower bound on $d_n$ known prior to this work is \(\Omega(\phi^n)\) for $\phi=(1+\sqrt5)/2$ given by a construction due to Paseman \cite{paseman}.  For $n\le19$, the best known lower bounds are given by Zivkovic \cite{smallest_not_in_range} and are optimal for $n\le9$.

Note that one has $|\detset|\ge 2d_n-1$ by observing $\detset\supset\{1-d_n,\cdots,d_n-1\}$.  The construction of Paseman thus guarantees $|\detset|=\Omega(\phi^n)$.  However, a stronger result on $|\detset|$ is known.  In particular, Tikhomirov \cite{singularity_bound} recently determined the related quantity
\[|\{M\in\{0,1\}^{n\times n}\,|\,\det(M)=0\}|=2^{n^2}\left(\frac12+o(1)\right)^n.\]
His technique involved showing that for a uniformly randomly selected $M$, there is strong anti-concentration of $\langle \mathbf{h},\mathbf{r}_n\rangle$, where $\mathbf{h}$ is a unit vector perpendicular to the first $n-1$ rows of $M$ with $n$th row $\mathbf{r}_n$.  In particular, one has for any $t$ that $P(\det (M) = t)\leq \left(\frac12+o(1)\right)^n$.  This gives a lower bound of $(2-o(1))^n$ on the size of the support of $\det M$, i.e. $$|\detset|\ge(2-o(1))^n.$$
The probabilistic approach is non-constructive, however, and it is unclear if it can
help determine any particular members of $\detset$.
This work provides a construction that guarantees $d_n\ge(2-o(1))^n$, which asymptotically improves the best known lower bound on $d_n$ and slightly improves the best known lower bound on $|\detset|$ by shrinking the $o(1)$ term so that the overall new bound is $d_n=\Omega(2^n/n)$ versus the old bound of $d_n\ge 2^n/g(n)$ for a particular $g(n)=\omega(n^2)$.

Let $\mathcal{M}(\row_2,\cdots,\row_n)\subset\{0,1\}^{n\times n}$ be the family of binary matrices whose $i$th row is given by $\row_i$ for each $i\in\{2,\cdots ,n\}$.  Note that there are $2^n$ matrices in such a family.  Our task will be to select suitable rows $\row_i$ such that taking the determinant of each matrix in the family will result in few collisions.  It will be easier to first construct rows $\srow_i\in\{-1,0,1\}^n$, and from those construct $\row_i\in\{0,1\}^n$.

\section{Lemmas}

\begin{lemma}
Let $\mathbf{r}_1,\cdots,\mathbf{r}_n$ be the rows of invertible square matrix $M$.  If
$\mathbf{v}\in\mathbb{R}^n$ is orthogonal to $\mathbf{r}_2,\cdots,\mathbf{r}_n$
with $\mathbf{v}(k)\neq0$ for some $k$,
then $\det (M) = c\,\sum_j \mathbf{v}(j)\mathbf{r}_1(j)$ where
\begin{equation}
\label{constantc}
c=\vecv(k)^{-1}\det\left(
\begin{bmatrix}
\ldash e_k\rdash\\
\ldash\row_2\rdash\\
\vdots\\
\ldash\row_n\rdash\\
\end{bmatrix}
\right)
\end{equation}
where $e_k$ is the $k$th row of the identity.
Importantly, $c$ does not depend on $\row_1$.
\end{lemma}
\begin{proof}
Both $\mathbf{v}$ and the first column of $M^{-1}$ are orthogonal to $\mathbf{r}_2,\cdots,\mathbf{r}_n$.  Since $\mathbf{r}_2,\cdots,\mathbf{r}_n$ are linearly independent, the orthogonal complement of their span has dimension 1.  Therefore, $\mathbf{v}$ is a scalar multiple of the first column of $M^{-1}$.  Recall that we can write $M^{-1}$ in terms of the cofactor matrix $C$ of $M$.  Specifically, $M^{-1}=\frac{1}{\det(M)}\,C^T$.  Thus the first column of $M^{-1}$ is a scalar multiple of the first row of $C$.  This allows us to claim that there is some $c$ such that $C_{1j}=c\mathbf{v}(j)$ for all $j\in[n]$.  By definition, $C_{1j}$ does not depend on $\row_1$, so $c$ does not depend on $\row_1$ either.  The Laplace expansion of the determinant gives $\det(M)= \sum_j C_{1j}\mathbf{r}_1(j)= c\sum_j \mathbf{v}(j)\mathbf{r}_1(j)$.  The formula for $c$ follows immediately by setting $\row_1=e_k$.
\end{proof}

\begin{lemma}
Let $1=b_1,\cdots,b_n$ be an integer sequence such that $b_{i+1}\leq b_1+\cdots+b_{i}$ for $i\in [n-1]$.  Then for every nonnegative integer $a\leq b_1+\cdots+b_n$, there exists $S\subset [n]$ such that $a=\sum_{i\in S} b_i$.
\end{lemma}
\begin{proof} This is proved using induction.  Assume it is true for $n-1$.  Then, if $a\leq b_1+\cdots+b_{n-1}$ we are already done.  We thus restrict our attention to $a>b_1+\cdots+b_{n-1}\ge b_n$.  In this case, note that $a-b_n\leq b_1+\cdots+b_{n-1}$, so applying the lemma for $n-1$ on $a-b_n$ again gives the result.
\end{proof}

\begin{lemma}
Let $\row_2,\cdots,\row_n\in\{0,1\}^n$ be linearly independent such that
\[
D:=\det\left(
\begin{bmatrix}
1&&\\
\cdash&\row_2&\cdash\\
&\vdots&\\
\cdash&\row_n&\cdash\\
\end{bmatrix}
\right)=\pm1.\]
If $\mathbf{v}\in\mathbb{R}^n$ is orthogonal to $\mathbf{r}_2,\cdots,\mathbf{r}_n$, and $\mathbf{v}(1),\cdots,\mathbf{v}(m)$ satisfy $0\le\mathbf{v}(i+1)\le\mathbf{v}(1)+\cdots+\mathbf{v}(i)$ for $i\in [m-1]$ and $\vecv(1)=1$, then
$d_n>\vecv(1)+\cdots+\vecv(m)$ where $d_n$ is the smallest natural number not in $\detset$.
\end{lemma}
\begin{proof}
If $D=-1$, then swap $\row_2$ and $\row_3$ so that $D=1$ and $\vecv$ is still orthogonal to all the rows.
For any invertible $M\in\mathcal{M}(\mathbf{r}_2,\cdots,\mathbf{r}_n)$, let $\row_1$ be its top row.  Then by Lemma 2.1, we can write for the same $c$ that
\[\det(M)=c\,\sum_{j=1}^n \vecv(j)\row_1(j)=c\,\sum_{j\in S}
\vecv(j)
\quad\text{where }S=\{j\,|\,\row_1(j)=1\}
\]
In particular, since $\vecv(1)\neq0$, Lemma 2.1 gives $c=D/\vecv(1)=1$.  By Lemma 2.2, we can select $\row_1$ so that $\det(M)=a$ for any positive integer $a\leq \mathbf{v}(1)+\cdots+\mathbf{v}(m)$.
\end{proof}

\section{Construction}
We will first construct a matrix $M$ whose rows $\mathbf{s}_i$ have entries in $\{-1,0,1\}$.  Then we will construct a transformation $T$ such that the rows $\row_i$ of $TM$ have entries in $\{0,1\}$.  Finally we will find $\mathbf{v}$ that will satisfy the hypothesis of Lemma $2.3$ along with the rows $\row_i$ of $TM$.

Fix an integer $k\ge 2$.  Let the top row of $M$ be $[1,0,\cdots,0]$.  We separate the rest of the rows of $M$ into two categories.
\begin{align*}
\text{`Recursive' rows:}&\quad\quad
\mathbf{s}_{i}(j)=
\begin{cases}
-1&\text{if }j=i\\
 1&\text{if }i-k\le j<i\\
 0&\text{o.w.}
\end{cases}
&\text{for }i\in\{1,\cdots,n-k\}.
\\
\text{`Finishing' rows:}&\quad\quad
\mathbf{s}_{i}(j)=
\begin{cases}
 1&\text{if }j=i\\
 1&\text{if }i-k\le j\le n-k\\
 0&\text{o.w.}
\end{cases}
&\text{for }i\in\{n-k+1,\cdots,n\}.
\end{align*}
The transformation $T=[t_{ij}]$ is defined by
\[
t_{ij}=\begin{cases}
 1&i=j=1\\
 1&\text{if }j\ge i>1\text{ and } i\equiv j \mod k\\
 0&\text{o.w.}
\end{cases}.\]
For example, for $k=3,n=10$ we have
\[M=\begin{bmatrix}
1\\
1&-1\\
1&1&-1\\
1&1&1&-1\\
&1&1&1&-1\\
&&1&1&1&-1\\
&&&1&1&1&-1\\
&&&&1&1&1&1\\
&&&&&1&1&&1\\
&&&&&&1&&&1\\
\end{bmatrix},
\quad
T=\begin{bmatrix}
1&&&&&&&&&\\
&1&&&1&&&1&&\\
&&1&&&1&&&1&\\

&&&1&&&1&&&1\\
&&&&1&&&1&&\\
&&&&&1&&&1&\\

&&&&&&1&&&1\\
&&&&&&&1&&\\
&&&&&&&&1&\\

&&&&&&&&&1\\
\end{bmatrix}.
\]

Note that the top row of $T$ is $[1,0,\cdots,0]$, so the top row of $TM$ is the same as the top row of $M$, which is also $[1,0,\cdots,0]$.
It is worth remarking that the case of $k=2$ is essentially equivalent to Paseman's construction.

\begin{lemma}
The rows $\mathbf{r}_i$ of $TM$ have entries in \{0,1\}.
\end{lemma}
\begin{proof}
First we noted above that $\row_1=[1,0,\cdots,0]$.  Then for $i\ge2$, the definitions of $T$ and $M$ give
$$\row_i=\srow_i+\srow_{i+k}+\cdots+\srow_{i+k\lfloor\frac{n-i}{k}\rfloor}.$$
We show $\row_i(j)\in\{0,1\}$ in each of several cases.
For $i\ge n-k$ we just have $\row_i=\srow_i\in\{0,1\}^n$ so are already done in that case.
For $j\ge n-k$, note that the last $k$ columns of $M$ match the last $k$ columns of the identity, so the last $k$ columns of $TM$ exactly match those of $T$, so all those entries are in $\{0,1\}$.

Now fix any $i,j\le n-k-1$.
Note that the support of $\mathbf{s}_{i+lk}$ is a subset of $\{i+(l-1)k,\cdots,i+lk\}$.
We consider the case of $j=i+lk$ for some $l$ and $j\in\{i+(l-1)k+1,\cdots,i+lk-1\}$ for some $l$ separately.
For $j=i+lk$ we have
$$\row_i(j)=\srow_i(j)+\srow_{i+k}(j)+\cdots+\srow_{i+k\lfloor\frac{n-i}{k}\rfloor}(j)
=
\srow_{i+lk}(j)+\srow_{i+(l+1)k}(j)=-1+1=0.$$
For $j\in\{i+(l-1)k+1,\cdots,i+lk-1\}$, we note that $j$ lies in the support of $\srow_{i+lk}$ and no other rows. So $\row_i(j)=1$.

\end{proof}

\begin{lemma}
If $\vecv$ is orthogonal to $\srow_2,\cdots,\srow_n$, then it is orthogonal to $\row_2,\cdots,\row_n$.
\end{lemma}
\begin{proof}
By hypothesis, $M\vecv=[\srow_1\cdot\vecv, 0,\cdots,0]^T$.  Since $T$ is upper triangular, $T[\srow_1\cdot\vecv, 0,\cdots,0]^T=[\srow_1\cdot \vecv, 0,\cdots,0]^T$ so $TM\vecv=[\srow_1\cdot\vecv, 0,\cdots,0]^T$ as required.
\end{proof}

Before we proceed, we define the $k$-step Fibonacci sequence.
\begin{definition}
Let $F_k(j)$ be the $j$th term of the $k$-step Fibonacci sequence.  That is, let $F_k(j)=0$ for $j\le0$, then let $F_k(1)=1$, and finally let $F_k(j)=\sum_{j'=j-k}^{j-1} F_k(j')$ for $j\ge3$.
\end{definition}
\begin{theorem}
For any $n\ge 2k$,
$$d_n>F_k(1)+\cdots+F_k(n-k).$$
\end{theorem}
\begin{proof}
We start by constructing $\vecv$ that is orthogonal to $\srow_2,\cdots,\srow_n$.  Let $\vecv(1)=1$.
Orthogonality with the `recursive' rows requires
\[\sum_{j=\max(i-k, 1)}^{i-1} \vecv(j)=\vecv(i)\quad\text{for }i\in\{2,\cdots,n-k\}.\]
Orthogonality with the `finishing' rows requires
\[-\sum_{j=i-k}^{n-k} \vecv(j)=\vecv(i)\quad\text{for }i\in\{n-k+1,\cdots,n\}.\]
Note each entry in $\vecv$ is defined solely in terms of the entries before it.  Further note that the definitions of $\vecv(j)$ and $F_k(j)$ match for $j\le n-k$, so we actually have
\[\vecv(j)=F_k(j)\text{ for }j\le n-k.\]
Then since all the terms are positive, $\vecv$ satisfies $\vecv(i)\le\vecv(1)+\cdots+\vecv(i-1)$ for $i\le n-k$.  By Lemma $3.2$, we have that $\vecv$ is orthogonal to $\row_2,\cdots,\row_n$.  
Since $M$ and $T$ are triangular with $\pm1$ on the diagonals, we have $\det(TM)=\pm 1$.
The hypotheses of Lemma $2.3$ are thus satisfied for $m=n-k$, so we have
$$d_n>\vecv(1)+\cdots+\vecv(n-k)=F_k(1)+\cdots+F_k(n-k)$$
as desired.
\end{proof}
We finish with an approximation of $F_k(j)$.
\begin{lemma}
$F_k(n)>\frac15\alpha_k^n$ for all $k\ge2,n\ge8$ where $\alpha_k$ is the zero of $z-2+z^{-k}$ closest to $2$.  Furthermore, $\alpha_k\in[2-2^{1-k},2)$.
\end{lemma}
\begin{proof}
An exact formula for $F_k(j)$ is known:
\begin{equation}
    \label{fib}
F_k(j)=\biggr\lfloor\alpha_k^{n-1}
\frac{\alpha_k-1}{k(\alpha_k-2)+\alpha_k}\biggr\rceil.
\end{equation}
We can approximately locate $\alpha_k$ using Rouche's theorem.  Let $K=\{z\in\mathbb{C}\,;\,|z-2|\le 2^{1-k}\}$.  Then on $\partial K$ we have $|z^{-k}|\le(2-2^{1-k})^{-k}<2^{1-k}=|z-2|$.  The number of zeros of $z-2$ inside $K$ is one (it is $z=2$), so $z-2+z^{-k}$ has exactly zero inside $K$ as well.
This implies $|\alpha_k-2|\le 2^{1-k}$.  Finally $z-2+z^{-k}$ is negative at $z=1.5$ and positive at $z=2$, so $\alpha_k$ is real and less than $2$.

The coefficient on $\alpha_k^{n-1}$ in (\ref{fib}) is decreasing in $\alpha_k$, so we can bound it by
\[\frac{\alpha_k-1}{k(\alpha_k-2)+\alpha_k}>
\frac{2-1}{k(2-2)+2}
=\frac12.\]
Thus
\[
F_k(n)>\frac12\alpha^{n-1}_k-1> \frac14\alpha^n_k-1.
\]
For $k\ge 2,n\ge 8$ we have $\alpha_k^n>20$, i.e. $\frac14\alpha^n_k-1>\frac15\alpha^n_k$.

\end{proof}
\begin{corollary}
\[d_n>c\,2^n/n\]
for $c=1/201$.
\end{corollary}
\begin{proof}
When $n<8$, the lower bound is less than 1 so the bound is trivially true. For $n\ge8$, we take the logarithm of $d_n$ and apply Lemma $3.4$.
\begin{align*}
\log (d_n)
  &\ge \log\left(F_k(1)+\cdots+F_k(n-k)\right)
\\&\ge \log F_k(n-k)
\\&\ge (n-k)\log\alpha_k - \log 5
\\&=   n\left(\left(1-\frac{k}{n}\right)\log\alpha_k - \frac1n\log 5\right)
\\&\ge n\left(\left(1-\frac{k}{n}\right)\log\left(2-2^{1-k}\right) - \frac1n\log 5\right)
\\&=   n\left(\log\left(2-2^{1-k}\right)-\frac{k\log\left(2-2^{1-k}\right)+\log 5}{n}\right)
\end{align*}
Set $k=\lfloor\log_2 n\rfloor$.  Then one has
\[\log\left(2-2^{1-\lfloor\log_2n\rfloor}\right)\ge\log2-\frac3{n}\]
and for $\epsilon=\log(10e^3)/\log n$, one has
\begin{align*}
\frac{3+k\log\left(2-2^{1-k}\right)+\log 5}{n}
\le
\frac{3+k\log2+\log5}{n}
\le\frac{\log n+\log(10e^3)}{n}
=(1+\epsilon)\frac{\log n}{n}.
\end{align*}
Thus
$\log (d_n)\ge n\log2-(1+\epsilon)\log n$.  Exponentiating gives $d_n\ge2^n/n^{1+\epsilon}$, and using the numerical approximation $n^{\epsilon}=10e^3<201$ yields the final result.
\end{proof}
\section{Acknowledgements}
The author thanks Asaf Ferber for his suggestion of this topic, funding, advisement, and guidance.

\newpage

\bibliography{mybib.bib}
\bibliographystyle{plain}
\end{document}